\documentclass[10pt]{amsart}
\usepackage[utf8x]{inputenc}

\usepackage{dsfont}
\usepackage{amsmath}
\usepackage{amssymb}
\usepackage{graphicx}
\usepackage{amssymb}
\usepackage{amsfonts}
\usepackage{soul}
\usepackage{mathpazo}
\usepackage{color}
\usepackage{yfonts}
\usepackage{paralist}
\usepackage{stmaryrd}
\usepackage{amsxtra}
\usepackage{latexsym,mathrsfs}

\usepackage{mathabx}

\usepackage{epsfig, enumerate}
\usepackage{color}
\usepackage{hyperref}

\newtheorem{theorem}{Theorem}
\newtheorem*{lemma*}{Claim}
\newtheorem{lemma}[theorem]{Lemma}
\newtheorem{coro}[theorem]{Corollary}
\newtheorem{definition}{Definition}

\newtheorem{prop}[theorem]{Proposition}

\numberwithin{theorem}{section}
\numberwithin{equation}{section}

\title[Colour Isomorphic Pairs]{On Color Isomorphic Pairs in Proper Edge Colourings of Complete Graphs}

\author{Xiao-Chuan Liu}
\address[Liu]{Instituto de Matem\' atica e Estat\' istica da Universidade de S\~ao Paulo,
R. do Mat\~ ao, 1010 - Vila Universitaria, S\~ ao Paulo, Brasil} 
\email{lxc1984@gmail.com}

\author{Xu Yang}
\address[Yang]{Instituto de Computação da Universidade Federal de Alagoas,
	Av. Lourival Melo Mota, S/N, Maceió, Brasil} 
\email{yang@ic.ufal.br}

\begin{document}
\maketitle{}
\begin{abstract} Following the recent paper \cite{conlon2021repeated} which initiated the study of 
colour isomorphism problems for complete graphs, we obtain upper bounds for 
$f_2(n,H)$ for a family of graphs $H$ obtained 
as the $K_0$-th rooted power of a balanced rooted tree for some sufficiently large $K_0$. The proof uses the random polynomial method of Bukh. We also obtain matching lower bounds for $1$-subdivisions of the complete bipartite graph.
\end{abstract}

\section{Introduction}
Relating to the study of rainbow patterns in proper edge colourings of complete graphs, 
Conlon and Tyomkyn recently in~\cite{conlon2021repeated}
proposed an extremal problem concerning colour isomorphic copies in proper edge colourings of complete graphs. 
Fixing a graph $H$, any $k$ vertex disjoint copies of $H$ in a properly edge coloured complete graph $K_n$
are called colour isomorphic if they are mapped to each other via a graph isomorphism preserving all the edge colours.

\begin{definition}[Defintion 1.1 of \cite{conlon2021repeated}]\label{colour_definition}
For any graph $H$ and an integer $k\geq 2$ fixed, 
write $f_k(n,H)$ to denote the smallest integer number $k_0$ so that there exists a proper colouring of $E(K_n)$ with $k_0$ colours which contains no 
$k$ vertex disjoint copies of $H$ in $K_n$ which are pair-wise colour isomorphic.
\end{definition}

This is indeed an extremal problem (see Section 2 of~\cite{conlon2021repeated}). 
Similar to other extremal problems in graph theory, 
the main objective in this topic is to determine the value or the order of $f_k(n,H)$ for various graphs $H$. 
In this paper, we are particularly interested in the $f_2$ function, which is the largest one among all $f_k$ functions by definition. 

\begin{definition} We say $T$ is a rooted tree, if there is an independent vertex subset $R\subset V(T)$, which we call the set of roots. We usually write $r=|R|$ to denote its size, and sometimes also write down $R=\{u_1,\cdots,u_r\}$ to be more precise.
In any graph $G$, when we fix a set $X=\{x_1,\cdots,x_r\}$, we say a labelled 
copy of $T$ is rooted at $X$, if the embedding of $T$ into $G$ sends each $u_i$ to $x_i$, for all $i=1,\cdots,r$. 
\end{definition}
\begin{definition}\label{density_definition}
The density of a rooted tree $T$ with its root set $R$, is defined as 
$\rho_T= \frac{e(T)}{v(T)-|R|}$. 
We say the rooted tree $T$ is balanced if the following holds. For any subset $S\subset V(T)\backslash R$, 
write $e(S)$ to denote the number of edges for which at least 
one endpoint belongs to $S$, then $\frac{e(S)}{|S|}\geq \rho_T$.
Define the \textit{$K$-th power of $T$}, written as $T^{K}$,
to be the union of $K$ labelled copies of $T$, which are all rooted at the same set $R$, satisfying that 
all the unrooted vertices are pair-wise disjoint.
\end{definition}

Let $T$ denote a rooted balanced tree, with $r$ roots, $b$ edges, and $a$ unrooted vertices.
Write $T^K$ to denote the \textit{$K$-th power of $T$}.
\begin{theorem}\label{main}
Let $T$ be a rooted balanced tree with density $\rho_T=\frac ba$, and assume $2a\geq b$. 
Then there exists a positive integer $K_0$, such that 
\begin{equation}
f_2(n, T^{K_0})=O(n^{\frac {2a}{b}}).
\end{equation}
\end{theorem}
The proof of this theorem is based on the 
random polynomial method of Bukh in~\cite{Bukhrandom} (see also~\cite{Conlon2019graphs} and \cite{bukh2018rational}).
In fact, similar to the result in~\cite{bukh2018rational}, 
Theorem~\ref{main} can also be generalized to a family version. 
Moreover, as an important step in~\cite{bukh2018rational} to attack a family version of the rational exponent conjecture, 
the authors show that for any rational number $r\in (0,1)$, there exists an example of a rooted tree such that $r=\frac{a}{b}=\frac{1}{\rho_T}$ (see Lemma 1.1 of~\cite{bukh2018rational}).
It follows that for the extremal problem of $f_2(n,H)$,
all rationals $r' \in (1,2)$ can be realized by some rooted tree with $r'=\frac{2a}{b}$.

We also have the following corollary. 
\begin{coro}\label{2a<b}
 Let $T$ be a balanced rooted tree with density $\rho_T=\frac ba$, and assume $2a \leq b$. 
Then for sufficiently large $K_0$, 
\begin{equation}
f_2(n, T^{K_0})=\Theta(n).
\end{equation}
\end{coro}
Note this gives the correct order for all complete bipartite graphs $K_{s,t}$ with $s\geq 2$ and 
$t$ sufficiently large. 
It was known that (see Theorem 1.3 (iii) of~\cite{conlon2021repeated}) for a bipartite graph $H$,  
\begin{equation}\label{2v-2_thm}
f_2(n,H) = \left\{
        \begin{array}{ll}
            O(n), &  \text{ if } e(H)\geq 2v(H)-2, \\
            O\big(n^{\frac{2v(H)-2}{e(H)}}\big), & \text{ if } e(H)\leq 2v(H)-2.
        \end{array}
    \right.
\end{equation}
So the above theorem shows in the family of large powers of balanced rooted trees with density $2$  and 
with $r\geq 2$,
 the first item as in the above statement happens with $e(H)=2v(H) -2r$. 

In the next theorem, we are going to obtain the matching lower bounds for 
another family of graphs, 
that is, $1$-subdivisions of complete bipartite graphs $K_{s,t}^{\text{sub}}$. 
\begin{theorem}\label{f2_lowerbound} 
For any $s\geq 2$ fixed, and for any 
$t$ sufficiently large depending on $s$, 
$f_2(n, K_{s,t}^{\text{sub}})=\Theta(n^{1+\frac 1s})$.
\end{theorem}
These numbers $1+\frac 1s$ seem to be the first known exponents between $1$ and $2$ for the $f_2$ functions. 
Below, in Section 2, we give the basic notation and gather a few lemmas. In Section 3, we prove Theorem~\ref{main} and Corollary~\ref{2a<b}, 
and in Section 4, we prove Theorem~\ref{f2_lowerbound}. In Section 5, we make some further comments on the connection with the classical Tur\'an number of bipartite graphs, providing one problem and one conjecture. 

\section{Preliminaries}
In this section, we fix notation and 
give a few lemmas introducing some previous relevant results.

Let $K_n$ denote the complete graph on $n$ vertices. A \textit{proper edge colouring} with some colour set $Q$
 is an assignment  $\chi: E(K_n)\to Q$ 
 so that no two adjacent edges are assigned with the same colour.
We also consider bounded edge colourings in the following sense. 
\begin{definition} We say a colouring of $E(K_n)$ is $C$-bounded for a constant $C$, 
if each colour class defines a graph on the vertex set $V(K_n)$ for which every vertex has degree at most $C$. 
\end{definition}
The following lemma follows quickly from Vizing's theorem on graph edge colourings. 
\begin{lemma}\label{bounded_lemma}[Proposition 2.4 of \cite{conlon2021repeated}]
 For any graph $H$, if $E(K_n)$ admits some 
 $C$-bounded colouring with $q$ colours containing no two colour isomorphic and vertex disjoint copies of $H$, then 
 it also admits some proper edge colouring with $(C+1)q$ colours containing no two colour isomorphic and vertex disjoint copies of $H$. 
 \end{lemma}

The next standard lemma will be needed for reducing a graph to a more regular one. 
\begin{lemma}\label{almost_regular}[essentially Theorem 1 of~\cite{erdos1965some}, 
see also Proposition 2.7 of~\cite{Jiang2012turan} or Lemma 2.3 of~\cite{conlon2021more}]
 Let $\varepsilon\in (0,1)$. 
 Suppose $G$ is a graph with $n$ vertices and with $e(G) =\omega (n^{1+\varepsilon})$.
Then one can find a subgraph $G'$ of $G$, with the following properties. 
\begin{enumerate} 
\item $G'$ is bipartite on a bipartition $A\cup B$.
\item for some $m$ tending to infinity when $n$ tends to infinity, 
we can choose $|A|=m$ and $|B|=\Theta(m)$. 
\item for some $\delta= \omega(m^{\varepsilon})$, and some absolute constant $\Delta$, 
for every vertex $v\in G'$, $\text{deg}_{G'}(v)\in [\delta,\Delta \delta]$.
\end{enumerate}
Moreover, a graph satisfying items (2) and (3) is called a balanced $\Delta$-almost regular graph. 
\end{lemma}

By definition, the extremal number of a graph $H$, expressed as $\text{ex}(n,H)$,
is the maximal number of edges an $n$-vertex graph $G$ can possibly contain, 
so that $H$ is not contained in $G$ as a subgraph. 
For $2\leq s\leq t$, write $K_{s,t}^{\text{sub}}$ to denote the $1$-subdivision of complete bipartite graph $K_{s,t}$. 
A relevant result for the extremal numbers is the following. 
\begin{lemma}[Theorem 1.8 in~\cite{conlon2021more}]\label{Kst'turan} Let $2\leq s\leq t$, 
then $\text{ex} (n, K_{s,t}^{\text{sub}})=O(n^{\frac 32-\frac{1}{2s}})$.
\end{lemma}
We will not directly apply this result, while the connection between this theorem and our Theorem~\ref{f2_lowerbound}
will become clear in Section~\ref{lower_bounds_Kst'}. See also Section~\ref{final} for more discussions. 

Next, we will need some basic facts about affine varieties over finite fields $\Bbb F_q$. 
Write $\overline {\Bbb F}_q$ to denote  the 
algebraic closure of $\Bbb F_q$ which is an infinite discrete field. 
Consider 
the $t$-dimensional vector spaces $\Bbb F_q^t$ and $\overline{\Bbb F}_q^t$.
An \textit{algebraic variety} 
$\mathcal W$ in $\overline {\Bbb F}_q^t$
over $\overline{\Bbb F}_q$ is defined as the zero set of $s$ polynomials $f_1,\cdots,f_s$ for some positive integer $s$, where each $f_i: \overline{\Bbb F}_q^t\to \Bbb F_q$ is a $t$-variable polynomial, and $s$ is a positive integer.  
We say $\mathcal W$ is \textit{defined} over 
$\Bbb F_q$ if in each of the defining polynomials, the coefficients are taken from $\Bbb F_q$. In this case, 
we can also consider the restricted variety 
$\mathcal W(\Bbb F_q)= \mathcal W\bigcap \Bbb F_q^t$.
We say $\mathcal W$ has \textit{complexity} $M$ if 
$s,t$ and the degrees of these $s$ polynomials are all bounded from above by the constant $M$.

\begin{lemma}\label{random_polynomial}[Lemma 2 in~\cite{Conlon2019graphs}]
Let  $d \geq m-1$, let $q$ be sufficiently large and write 
$\mathcal P_d$ to denote the collection of all the $t$-variable polynomials over the field $\Bbb F_q$ 
with degree at most $d$.
  Fix any $m$ distinct points $X_1,\cdots,X_m\in \Bbb F_q^t$. 
 If we choose $f\in \mathcal P_d$ uniformly at random, then  
  \begin{equation}
\textbf{Prob} \big[f(X_i)=0 \text{ for every } i=1,\cdots m\big] = \frac{1}{q^m}.
 \end{equation}
\end{lemma}
 The next lemma is a consequence of the classical 
 Lang-Weil bound (see~\cite{lang1954number}) which relates the cardinality of 
 a variety which is defined over the finite field $\Bbb F_q$ with the dimension of the variety. 
 We will even avoid the formal mention of the notion of dimension in this context. Instead, 
 the following statement was taken directly from~\cite{bukh2018rational}. 
 It was adopted to be easier 
 to use in the context of our problems.  
\begin{lemma}[cf. Lemma 2.7 in~\cite{bukh2018rational}]
\label{Langweil} Let $\mathcal W$ and $\mathcal W'$ be two varieties in the vector space 
$\overline{\Bbb F}_q^t$, over $\overline {\Bbb F}_q$, each of which has complexity at most $M$. Write 
$\mathcal W(\Bbb F_q) = \mathcal W\cap \Bbb F_q^t$ and 
$\mathcal W'(\Bbb F_q) = \mathcal W'\cap \Bbb F_q^t$.
Then there exists a constant $C=C(M)$ so that one of the following holds with $q$ being sufficiently large.
\begin{enumerate}
\item either $\mathcal W (\Bbb F_q)\backslash \mathcal W'(\Bbb F_q)$ 
has size at least $q/2$,
\item or its size is bounded by $C$.
\end{enumerate}
\end{lemma}

\section{Upper Bounds}
This section is mainly devoted to the proof of Theorem~\ref{main}. At the end we will also deduce Corollary~\ref{2a<b} from Theorem~\ref{main}. 

Recall that we are given a balanced tree $T$, with $r$ roots, a unrooted vertices and $b$ edges. 
Let us take a large prime number $q$, and consider the finite field $\Bbb F_q$. 
We will define a complete bipartite graph $G$ 
with randomly assigned colours to its edges via random polynomials. 
Firstly, let $G$ be the complete bipartite graph on the bipartition $U\bigcup V$, where $U$ and $V$ are two copies of $\Bbb F_q^b$. 
Fix a positive integer $d$ which is larger than $2rb^2+b$ 
(in order to apply Lemma~\ref{random_polynomial} later), 
and look at polynomials on $2b$-variables over the field $\Bbb F_q$, 
namely $(X_1,\cdots, X_b, Y_1,\cdots, Y_b)$,  and of degree at most $d$. 
Write $\mathcal P_d$ to denote the collection of all such polynomials. 

Take a random function  
$F=(f_1,\cdots,f_{2a}): \Bbb F_q^b\times F_q^b \to \Bbb F_q^{2a}$, 
where for each $i=1,\cdots,2a$, the function 
$f_i : \Bbb F_q^b\times \Bbb F_q^b \to \Bbb F_q$ is simply a polynomial chosen uniformly at random from $\mathcal P_d$. 
Therefore $F$ possibly takes $q^{2a}$ distinct values, which will be the set of colours we use. 
So we can think of $G$ as the complete bipartite graph whose edge are randomly coloured via these random polynomials. 

We need to induce this colouring back to a colouring of the edge set of the complete graph 
$K_n$ of size $n=q^b$. For this, we can identify $V(K_n)= [n]=\{1,\cdots, n\}$ with 
$\Bbb F_q^b$ with an arbitrary bijection $\phi: [n]\to \Bbb F_q^b$. 
In this way, $\Bbb F_q^b$ is endowed with an induced order. 
Now, for any edge $(i,j) \in E(K_n)$, 
 we map the edge $(i,j)$ to an edge of the bipartite graph $G$ so that with respect to the above natural order, 
 the smaller vertex is mapped into $U$ and the larger vertex is mapped into $V$.
More precisely, for any $(i,j)\in E(K_n)$,
\begin{equation}
\Phi((i,j)) = (\phi(\min\{i,j\} ),\phi( \max\{i,j\})) \in U\times V.
\end{equation} 

Clearly, $\Phi: E(K_n) \to E(G)$ is an injective map. 
Now the random coloured complete bipartite graph $G$ induces 
a colouring for $E(K_n)$ with $q^{2a}$ colours by assigning each edge $(i,j)$
the colour $F\circ \Phi\big( (i,j) \big)$.

Consider a pair $(T_1,T_2)$ of vertex disjoint labelled copies of $T$, which comes with 
a one-to-one correspondence between their edges.
Suppose also $(T_1, T_2)$ is rooted at the pair of disjoint vertex subsets $(R_1, R_2)$. 
Next, we consider $p$ such pairs, 
namely $(T_1^{(1)},T_2^{(1)}), \cdots, (T_1^{(p)},T_2^{(p)})$, each of which is 
rooted at the same pair $(R_1,R_2)$. 
Write $H_1=\bigcup_{j=1}^p T_1^{(j)}$ and $H_2=\bigcup_{j=1}^p T_2^{(j)}$. 
In other words, $(H_1,H_2)$ is a pair of graphs, obtained as follows. 
For the $p$ pairs as above, the graph $H_1$ is the union of $p$ copies of $T$ which are the first one from each of the pairs. Similarly, $H_2$ is the union of the $p$ copies of $T$ which are the second one from each 
pair. Clearly, $R_1\subset V(H_1), R_2\subset V(H_2)$. Furthermore,
these is an edge correspondence for $(H_1,H_2)$ induced from the $p$ pairs above. 
Caution that $H_1$ and $H_2$ can possible intersect at certain edges.

Now we explain with more details the edge correspondence for the pair $(H_1,H_2)$.
For any $j=1,\cdots, p$, 
let us list the edges of $T_1^{(j)}$ as $\{e_1,\cdots,e_b\}$ and 
list the edges of $T_2^{(j)}$ as $\{e_1',\cdots,e_b'\}$, 
where each pair $(e_\ell, e_\ell')$ can be mapped from one another via the natural identification, since they are two labelled 
copies of the same tree $T$. 
Below we say $(e_\ell, e_\ell')$ forms a corresponding pair.
Note that,
 in order to compute the probability that each of these pairs forms an colour isomorphic pair (or a monochromatic $2$-matching), 
we will need the following $b$ equations to be satisfied in the future. 
\begin{equation}\label{equations_constraints}
F(\Phi(e_\ell))-F(\Phi(e_\ell'))=0, \ell=1,\cdots, b.
\end{equation}
Therefore, we gather all these equations for $j=1,\cdots,p$, and for $\ell=1,\cdots,b$. 
They form a system of $bp$ equations. 
However, this system of equations might use certain edges for multiple times. 
In particular certain equations might be redundant. In order to compute the probability that all these $p$ pairs of $T$ copies to be colour isomorphic pairs, 
we need to understand in this system, essentially one needs to check how many polynomial equations to hold. 

In other words, for all the edges, we can think that 
being in a corresponding pair is an equivalence relation among edges, 
and the above paragraph means one can write down a collection of $bp$ equivalent pairs. 
We will be interested in 
the least number of these pairs which already contains all the information about this equivalence relation. 

To sum up, for a pair of graphs  $(H_1,H_2)$,
which is seen as the union of $p$ distinct pairs of labelled copies of $T$ sharing the same 
vertex disjoint pair of roots $(R_1,R_2)$, 
we obtain a natural $EC$ (short for \textit{edge correspondence})
between the two edge sets of these two graphs.
More precisely, one can find certain collection  of pairs $(e,e') \in H_1\times H_2$, 
which all together  indicate that in the future, 
each pair of edges must be assigned with the same colour.
Note in such a correspondence, each edge in each graph must show up at least once. 

Now we are ready to make the following important definition. 
\begin{definition}\label{constraint_number}
Consider a pair of graphs $(H_1,H_2)$ rooted at vertex disjoint pair of roots $(R_1,R_2)$, 
 with a specified edge correspondence $\text{EC}$ between their edge sets.
 Define its \textit{constraint number}, written as $k=k(H_1,H_2, 
 \text{EC})$\footnote{Here we omit the dependence on the root pair $(R_1,R_2)$, because usually the roots are clear to see from the context.},
 to be the smallest number of pairs needed to express all the corresponding relations given by $\text{EC}$.
\end{definition}
Translating back to the language of polynomial equations, when $\text{EC}$ is obtained from the $pb$ edge equivalence relations, $k(H_1,H_2, \text{EC})$ is simply the smallest integer $k$ so that, in the  
$bp$ equations described in (\ref{equations_constraints}), 
there exists a choice of $k$ equations therein which already determines the whole system.

For example, if $p=1$, 
then $(H_1,H_2)= (T_1,T_2)$ is a pair of vertex disjoint labelled copies of $T$.
In this case, the edge correspondence is very clear so we can omit it. 
It follows easily that 
$k(T_1,T_2)=b$ which is the number of edges in $T$. 
A bit more generally, if the two graphs $H_1$ and $H_2$ 
do not share any common edge, then for any edge correspondence $EC$ which is induced by 
$p$ pairs of vertex disjoint labelled copies of $T$,
it is not hard to see that
$k(H_1, H_2, \text{EC})\geq \max \{e(H_1),e(H_2)\}$, because in any set of equations which
can express all the edge equivalence relations, 
all the edges from $H_1$ or $H_2$ must appear at least once. 

See Figure~\ref{Figure} for some illustrations. In both examples, $T$ is a path of length $4$, which is rooted at two leaves, and $p=2$.
In the example depicted in the first line, both $H_1$ and $H_1$ are given by the union of two copies of $P_5$ which are vertex disjoin except at the roots. In this case, we have 
$k(H_1,H_2,\text{EC})=10=e(H_1)=e(H_2)$. 
Specifically, the $10$ edge equivalence relations are simply
	 $e_i \sim f_i$, $i=1,\ldots, 10$.
		In the example depicted in the second line, 
		$H_1$ is the union of two copies of $P_5$, namely, $e_1e_2e_3e_4e_5$ and $e_6e_4e_3e_2e_7$.
		$H_2$ is the union of two copies of $P_5$, namely, $f_1f_2f_3f_4f_5$ and $f_1f_6f_3f_7f_5$. 
		Here, $k(H_1,H_2,\text{EC})= 9>e(H_1)=e(H_2)$. The equivalence relations are as follows. 
		$e_i=f_i$, $i=1,\ldots, 5$, 
		$e_6=f_1$, $e_4=f_6$, $e_2=f_7$, $e_7=f_5$. 
\begin{figure}[h]\label{Figure}
	\centering
	\includegraphics[width=0.8\textwidth]{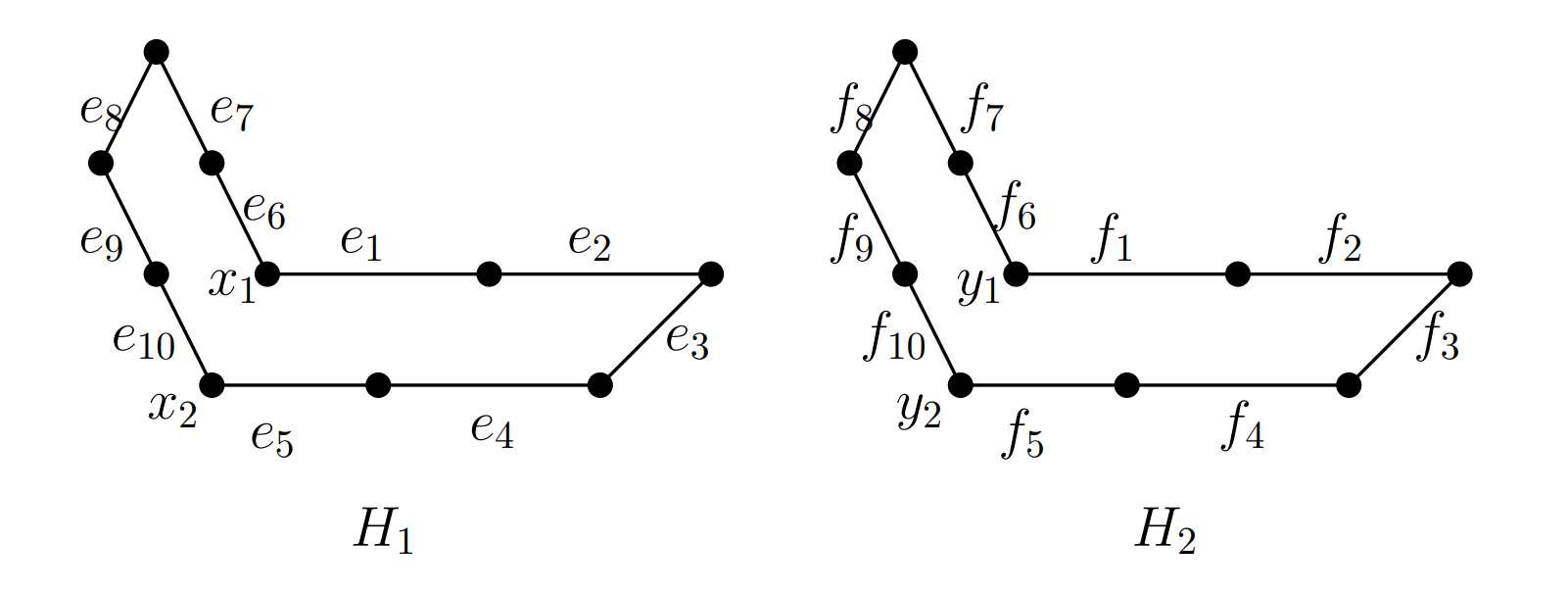}
	\includegraphics[width=0.8\textwidth]{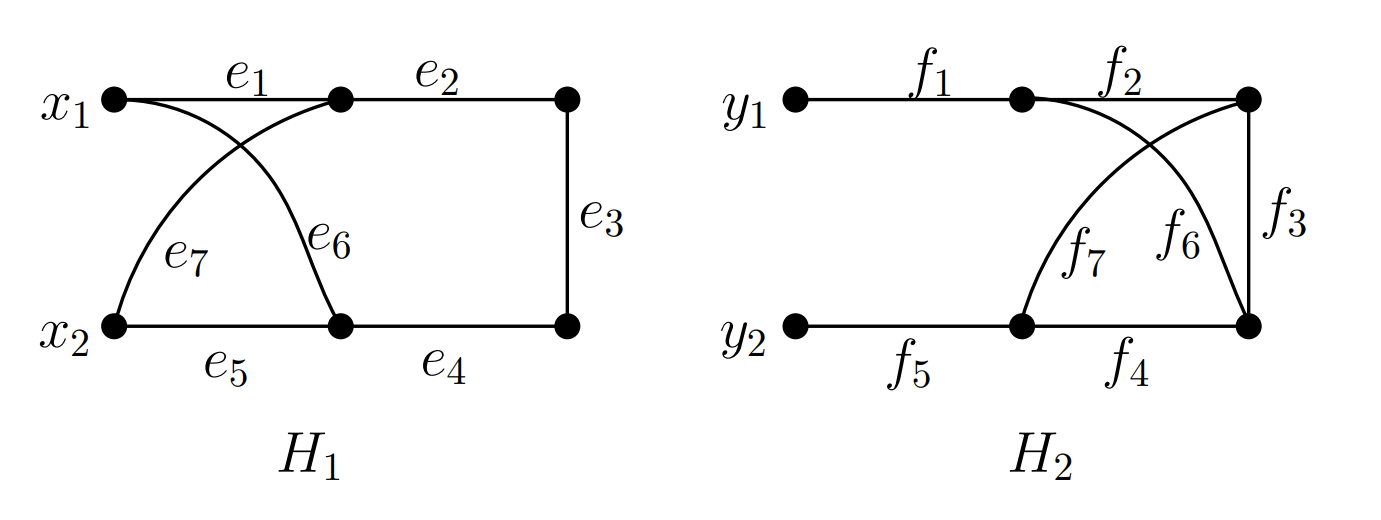}
	\caption{The example depicted in the first line
is non-degenerate, while the example depicted in the second line is degenerate. 
	}\label{figure}
\end{figure}

The next lemma is crucial in estimating higher moments 
of certain random variables later. 
\begin{lemma}\label{balance_two} 
 Let $p$ be a positive integer.  
  Suppose the edge correspondence $\text{EC}$ for $(H_1,H_2)$ 
 is induced from $p$ pairs of vertex disjoint copies of $T$, each of which 
 is rooted at $(X,Y)$. 
 Then the following hold.
\begin{align}\label{balance_for_two} 
                     2k \big(H_1,H_2,\text{EC} \big) \geq   \ \ 
                 &  \rho_T \cdot \big(v(H_1\bigcup H_2) -2r \big) \\
          =   \ \   &  \frac ba  \cdot \big(v(H_1 \bigcup H_2) -2r \big). \nonumber
\end{align}
\end{lemma}
\begin{proof} 
The proof goes by induction. By assumption,
 $X$ and $Y$ are two disjoint $r$-subsets of $V(K_n)$ 
 which are simply two labelled copies of the root set $R$ of $T$.
Let us first check the base case $p=1$.
In this case, the pair $(H_1,H_2)=(T_1,T_2)$ and $\text{EC}$ 
must come from the natural correspondence between 
edges of $T_1$ and $T_2$. 
So $k(T_1,T_2,\text{EC})=b$, $v(T_1\bigcup T_2)=2a+2r$ and therefore $(\ref{balance_for_two})$ holds with equality. 

Inductively, we now suppose the conclusion holds for the integer $p-1$. 
We can think of $(H_1, H_2)$ as formed by 
a pair $(H_1',H_2')$ which is formed by $p-1$ pairs of vertex disjoint labelled copies of $T$, 
together with another pair $(T_1,T_2)$ of vertex disjoint copies of $T$, rooted at the pair $(X,Y)$, 
so that $H_1=H_1'\bigcup T_1$ and $H_2=H_2'\bigcup T_2$.
Now we write $S_1=V(T_1) \backslash V(H_1'\bigcup H_2')$ and 
$S_2=V(T_2) \backslash V(H_1' \bigcup H_2')$. 
Then, recalling Definition~\ref{density_definition}, 
we can write $e(S_1)$ (respectively, $e(S_2)$) to denote the number of 
edges in $T_1$ (respectively, in $T_2$) which has at least one endpoint belonging to $S_1$ (respectively, $S_2$).
Then, by the assumption that $T$ is a balanced rooted tree, we have that, 
\begin{equation}\label{new_constraints}
2\max\{e(S_1),e(S_2)\} \geq e(S_1)+e(S_2) \geq \rho_T (|S_1| +|S_2|).
\end{equation}

By induction hypothesis, we have 
\begin{equation}\label{H_astestimate}
2k(H_1',H_2', \text{EC}) 
\geq \rho_T   \big(v(H_1' \bigcup H_2')  -2r \big)
\end{equation}
From the pair $(H_1', H_2')$ to the pair $(H_1,H_2)$, for the edge correspondences,
at least $\max\{e(S_1),e(S_2)\}$ new constraints are added so that 
the system of $bp$ equations (\ref{equations_constraints}) to be satisfied. 
Equivalently, 
\begin{equation}\label{k_increment}
k(H_1,H_2,\text{EC}) \geq k(H_1',H_2',\text{EC}) +\max\{e(S_1),e(S_2)\}
\end{equation}

Therefore, (\ref{balance_for_two}) follows by the combination of (\ref{new_constraints}), (\ref{H_astestimate}) and 
(\ref{k_increment}).
\end{proof}

\begin{prop}\label{no_color_iso_pair}
With probability converging to $1$ as $q$ tends to infinity, 
the random colouring obtained for $E(K_{q^b})$ via the random polynomial $F$ on $G$
contains no two vertex disjoint and colour isomorphic copies of $T^{K_0}$'s.
\end{prop}
\begin{proof}
Fix $X=(x_1,\cdots,x_r)$ and $Y=(y_1,\cdots,y_r)$, which are two $r$-subsets of $[n]$, point-wise disjoint. 
Define $\mathcal C= \mathcal C(X,Y)$ to be the collection of all pairs $(T_1,T_2)$ of colour isomorphic labelled
copies of $T$, which are rooted at $(X,Y)$. 
Then we are interested in estimating its size $|\mathcal C|$. 
Observe that, if there exists two colour isomorphic labelled copies of $K_0$-th 
power $T^{K_0}$, rooted at $(X,Y)$, 
 then it follows that $|\mathcal C|\geq K_0$.
 Therefore, it suffices to show that, with high probability, 
in the random colouring of $E(K_n)$, for any choice of $(X,Y)$, the size of the collection
$\mathcal C(X,Y)$ is smaller than $K_0$.

In order to estimate $|\mathcal C|$, we first estimate its $p$-th moment for any positive integer $p$. 
Note that $\mathcal C ^p$
can be interpreted as the collection of $p$-tuples of pairs (possibly repeating), namely, 
$\big ( (T_1^{(1)},T_2^{(1)}), \cdots, (T_1^{(p)},T_2^{(p)})  \big)$, such that 
\begin{enumerate}
\item for each $j=1,\cdots, p$, the pair $(T_1^{(j)},T_2^{(j)})$ is a pair of colour isomorphic 
and vertex disjoint labelled copies of $T$.
\item for each $j=1,\cdots, p$, the pair $(T_1^{(j)},T_2^{(j)})$ is rooted at $(X,Y)$.
\end{enumerate}
With this interpretation, we will count all these $p$-tuples with respect to all of their possible realizations.
So we define $\mathcal T^p(\mathcal C)$ to be the collection of pairs of graphs 
$(H_1,H_2)$ so that for some choice $\big ( (T_1^{(1)},T_2^{(1)}), \cdots, (T_1^{(p)},T_2^{(p)})  \big) \in \mathcal C^p$, 
we have 
$H_1= \bigcup_{j=1}^p T_1^{(j)}$ and $H_2=\bigcup_{j=1}^p T_2^{(j)}$. 

Moreover, for any fixed $(H_1,H_2) \in \mathcal T^p(\mathcal C)$ defined as above, they also come with 
a natural edge correspondence $EC$, as was explained earlier. 
Then we write 
$N_p(H_1,H_2,\text{EC})$ to denote the number of elements in $\mathcal C^p$ 
which satisfies that 
$H_1= \bigcup_{j=1}^p T_1^{(j)}$, $H_2=\bigcup_{j=1}^p T_2^{(j)}$,
and EC is induced by these $p$ pairs correspondingly. 
Clearly, 
\begin{equation}\label{Np_bound}
N_p(H_1,H_2,\text{EC}) = O\big ( q^{b \cdot ( v(H_1\cup H_2) -2r ) } \big).
\end{equation}
On the other hand, recall that among these $pb$ relations which defined EC, 
we can choose $k=k(H_1,H_2,\text{EC})$ from them which determine all these relations, and 
the number $k$ can not be smaller for this purpose. 
Consider the event that a specific $p$-tuple of pairs realising the pair $(H_1,H_2)$ and the prescribed EC
has been chosen to be all colour isomorphic pairs. By Lemma~\ref{random_polynomial}, provided that $d>pb$,
 the probability 
that such an event happens is at most 
 \begin{equation}
 q^{-2a k(H_1,H_2,\text{EC})}.
 \end{equation}
In order to upper bound $ \Bbb E(|\mathcal C|^p)$, 
we sum over all the triples $(H_1,H_2,\text{EC})$ that can possibly appear. Thus,
\begin{align}\label{p-th_moment_for_C}
       \Bbb E(|\mathcal C|^p) \leq  & \sum N_p(H_1,H_2, \text{EC}) q^{-2a k(H_1,H_2,\text{EC})} \\
\leq   & C \sum q^{b (v(H_1\cup H_2) -2r)} q^{-2a k(H_1,H_2,\text{EC})} = O_p(1), \nonumber
\end{align}
where the last inequality follows from (\ref{balance_for_two}).

It would be nice if we can relate the collection $\mathcal C$ to a certain algebraic variety 
 in order to apply 
the Lang-Weil bound. However this can only be done in an indirect way. 
For any labelled copy $T_1$ of $T$ found as a subgraph of $K_n$, define its \textit{shadow} 
as a subgraph of $G(U\bigcup V)$, where for every edge $(i,j)\in E(T_1)$, 
choose one and only one of the edges from either $(i,j)\in E(G)$ or $(j,i)\in E(G)$. 
Therefore, each copy $T_1$ has 
exactly $2^{b}$ many shadow graphs contained in $G$. 
Note that a shadow graph of $T_1$ may have more vertices but always has $b$ edges. 
We also say $T_1$ is the projection of $T_1'$ to $K_n$. 

For fixed $X$ and $Y$, we define $\mathcal D=\mathcal D(X,Y)\subset G^2$ 
as the collection 
such that, a pair $(T_1',T_2')$ of subgraphs of $G(U\bigcup V)$ belongs to $\mathcal D$
if and only if they are the shadows of an element of $\mathcal C=\mathcal C(X,Y)$.
 It turns out $\mathcal D$ can be interpreted as the difference set of two varieties. Let us explain below. 
  
 For definiteness, for the balanced rooted tree $T$, let  $u_1,\cdots,u_r$ denote the roots, 
  $v_1,\cdots,v_a$  denote the unrooted vertices of $T$, and $e_1,\cdots,e_b$ denote the edges. Now
 the fixed roots $x_1,\cdots,x_r,y_1,\cdots y_r$ are seen as distinct points in 
 $[n]$ which obtain a natural order.
 With respect to this order, 
 for any pair $(T_1,T_2)$ of vertex disjoint trees rooted at $(X,Y)$, 
 there is a certain \textit{order type} $\Sigma=(<_1, <_2)$, which is
 a pair of permutations on $[a+r]$, indicating how the vertices the two embedded copies 
 of $T$ are located with respect to natural order in $[n]$.
 It will be a bit cumbersome to actually write down these orders. 
We content ourselves by noticing that the number of all these order types with fixed $(X,Y)$
is at most $((r+a)!) ^2$. 

In a copy $T_1$ of $T$, with a prescribed order type say, $<_1$, 
for any edge $e=(w_i,w_j)$ we can write 
\[   
\Phi(e)= \Phi_{<_1} \big( (w_i, w_j) \big) = 
     \begin{cases}
      (w_i, w_j), & \quad \text{ if } w_i <_1 w_j, \\
      (w_j, w_i), & \quad\text{ if } w_j <_1 w_i. \\ 
     \end{cases}
\]

Now we consider the vector space $\Bbb F_q^{2ab}$, and write 
each element down with $2a$ coordinates, namely, 
$X_1,\cdots,X_a,Y_1,\cdots,Y_a$, 
where each $X_i$ or $Y_i$ represents an element from $[n] = V(K_n)$.
Fix $X=(x_1, \cdots, x_r)$ and $Y=(y_1, \cdots, y_r)$, 
and with respect to their order induced from $[n]$, we choose one compatible order 
$\Sigma =(<_1,<_2)$, discussed as in the previous paragraph. 
Define the variety $\mathcal W_\Sigma$ over $\overline{\Bbb F}_q$  as the zero set of 
the following $b$ polynomials in $\overline{\Bbb F}_q^{2ab}$, with respect to the $b$ edges of $T$. 

There are two kinds of equations, depending on the nature of an $e\in E(T)$. 
Firstly consider an edge $e\in E(T)$ which joins a root $u_i$ with some unrooted vertex $v_j$. 
Then corresponding edges in $T_1$ and $T_2$ are $(x_i,X_j)$ and $(y_i,Y_j)$.
We map them to the bipartite graph $G(U\bigcup V)$ via $\Phi_{<_1}$ and $\Phi_{<_2}$, respectively. 
In other words, $\Phi_{<1}(x_i,X_j)$ takes the image $(x_i,X_j)$ if $u_i<_1 v_j$, and 
it takes the image $(X_j,x_i)$ if $v_j<_1 u_i$ (the image of $\Phi_{<2}(y_i,Y_j)$ can also be obtained similarly). So 
we obtain the equation
\begin{equation}\label{equation_on_roots}
F \big (\Phi_{<_1}(x_i, X_j) \big)-F \big(\Phi_{<_2}(y_i,Y_j) \big)=0.
\end{equation}
Secondly consider an edge $e\in E(T)$ joining a pair of unrooted vertices $v_i$ and $v_j$. 
The corresponding pair of edges in $(T_1,T_2)$ are $(X_i,X_j)$ and $(Y_i,Y_j)$, respectively. 
Then we map them to the bipartite graph $G(U\bigcup V)$ via $\Phi_{<_1}$ and $\Phi_{<_2}$, recpectively.
In other words, $\Phi_{<_1}(X_i,X_j)$
 takes the image $(X_i,X_j)$ if $v_i<_1 v_j$ and it takes the image $(X_j,X_i)$
if $v_j<_1v_i$ (the image of $\Phi_{<2}(Y_i,Y_j)$ can also be obtained similarly). So we obtain the equation 
\begin{equation}\label{equation_without_roots}
F \big (\Phi_{<_1}(X_i,X_j) \big) - F \big(\Phi_{<_2}(Y_i,Y_j) \big)=0.
\end{equation}
These functions define a variety $\mathcal W_\Sigma$.
The union of all such varieties, $\mathcal W=\bigcup_{\Sigma}\mathcal W_{\Sigma}$, 
is itself a variety. Then we define $\mathcal W'$ to be the degenerate variety, which is defined by the union of the zero sets given by any one of the following equations. 
\begin{align}
& X_k=x_i \text { or } X_k=y_j,  \text{ for some } k=1,\cdots,a, \text{ and } i,j=1,\cdots,r.\\
& Y_k=x_i \text{ or } Y_k=y_j,\text{ for some } k=1,\cdots,a, \text{ and } i,j=1,\cdots,r. \\
& X_k=Y_\ell, \text{ for some }k,\ell=1,\cdots,a.
\end{align}
Then we claim that 
$|\mathcal D|  = |\mathcal W(\Bbb F_q) \backslash \mathcal W'(\Bbb F_q)|$. 
To check this, one notices that 
 every element in $\mathcal W(\Bbb F_q) \backslash \mathcal W'(\Bbb F_q)$ corresponds to a non-degenerate solution 
 of the equations (\ref{equation_on_roots}) and (\ref{equation_without_roots}), with respect to a certain order 
 $\Sigma=(<_1, <_2)$. But the order is irrelevant for now since we are working on the graph $G(U\bigcup V)$. 
 Clearly, each one such solution describes exactly a shadow graph 
 of some pair $(T_1,T_2)$ of copies of $T$'s found as subgraphs of $K_n$ rooted at $(X,Y)$. 
 In other words, it represents a pair of subgraphs of $G$, 
 whose projection to $K_n$ is an element of $\mathcal C$.
 We observe the basic relationship $|\mathcal D(X,Y)| = 2^{2b}|\mathcal C(X,Y)|$.
 
 We still need to estimate $|\mathcal D|^p$ for any positive $p$. 
 Note it counts the number of ordered $p$ tuples of pairs of shadows, which projects to an element of $\mathcal C^p$. 
 Take any element of $\mathcal D^p$, written down as 
 $(T_1^{(1)'}, T_2^{(1)'}),  \cdots, (T_1^{(p)'},T_2^{(p)'})$, such that
 \begin{enumerate}
 \item for each $j=1,\cdots,p$, the pair $(T_1^{(j)'}, T_2^{(j)'})$ obtains the same colour in the randomly coloured graph $G(U\bigcup V)$.
 \item for each $j=1,\cdots,p$, the pair $(T_1^{(j)'}, T_2^{(j)'})$ projects to a pair $(T_1^{(j)}, T_2^{(j)}))$ belonging to $\mathcal C$.
 \end{enumerate}
  Naturally we obtain an element $(T_1^{(1)}, T_2^{(1)}),  \cdots, (T_1^{(p)},T_1^{(p)})\in \mathcal C^p$.
 As in earlier arguments, it comes with a pair of graphs $(H_1,H_2)$ with an edge correspondence $EC$, 
 where $H_1=\bigcup_{j=1}^pT_1^{(j)}$ and $H_2=\bigcup_{j=1}^p T_2^{(j)}$.
  Define $N_p'(H_1,H_2,\text{EC})$ the number of elements of $\mathcal D^p$ which induces a fixed tripe $(H_1,H_2,\text{EC})$. 
 It follows from (\ref{Np_bound}) that
 \begin{equation}
 N_p'(H_1,H_2,\text{EC})=  2^{2b}   N_p(H_1,H_2,\text{EC}) = O\big (q^{b ( v(H_1\cup H_2) -2r ) } \big).
 \end{equation}
 We can again sum over all the possible triples $(H_1,H_2,\text{EC})$ which are obtained after the projections, as was done in (\ref{p-th_moment_for_C}). 
 
Now, in the $p$-tuple of shadow pairs, the number of edge relations that need to be checked is at least $k(H_1,H_2,\text{EC})$ and possibly larger than it. To see this, we choose $k$ edge pairs which can determine $EC$ for the $p$-tuple in $\mathcal C$. 
Then observe possibly one pair of edges in these equations might become two pairs in the graph $G(U\bigcup V)$. 
Then one must keep both in order for the system of equations to be determined. 
Consider the event that a specific $p$-tuple of pairs of shadows which induce a fixed pair $(H_1,H_2)$ with $EC$ in $K_n$. Then, by Lemma~\ref{random_polynomial} and provided $d>pb$, 
the probability that such an event happens is 
at most 
\begin{equation} q^{-2ak(H_1,H_2,\text{EC})}.
\end{equation}
Then, we can follow the lines of (\ref{p-th_moment_for_C}) to estimate that 
\begin{align}
       \Bbb E(|\mathcal D|^p) \leq  & \sum N_p'(H_1,H_2,\text{EC})q^{-2a k(H_1,H_2,\text{EC})} \\
\leq   & C \sum q^{b (v(H_1\cup H_2) -2r)} q^{-2a k(H_1,H_2,\text{EC})} = O_p(1). \nonumber
\end{align}
Finally, we can apply the Lang-Weil bound (Lemma~\ref{Langweil}). For some sufficiently large constant $K'$, and for any positive integer $p$.
\begin{align}
                    &  \textbf{Prob} \big(|\mathcal D| > K'\big) = \textbf{Prob}\big(|\mathcal W\backslash \mathcal W'| > K' \big )\\
                 = \ \ &  \textbf{Prob} \big(|\mathcal W\backslash \mathcal W'|^p \geq (\frac{q}{2})^p \big) \leq O_p(1) \cdot \frac {1}{q^p}, \nonumber
\end{align}
where the last step follows from Markov's inequality. 

Take $K_0=\frac{K'}{2^{2b}}$ and then define 
$B$ to be the collection of pairs $(X,Y)$ for which $|\mathcal C(X,Y)| > K_0$. Then the event 
$|B|>0$ implies that $|\mathcal D(X,Y)|> K'$ for at least some choice $(X,Y)$. 
Take 
$p=2rb+1$, and recall we have chosen that $d>2rb^2+b$. 
Applying the Markov inequality,  
\begin{equation}
\textbf{Prob}(|B|>0) \leq \Bbb E(|B|) \leq q^{2r b} \frac{O_p(1)}{q^p} = o(1),
\end{equation}
where the second inequality is by the union bound. The proof is completed.
\end{proof}

Next we show with high probability, the random colouring is bounded.  
\begin{prop}\label{bounded_whp}
  Suppose $2a\geq b$.
With probability converging to $1$ as $q$ tends to infinity, 
the random colouring obtained for $E(K_{q^b})$ via the random polynomial $F$
is $C_0$-bounded for some absolute constant $C_0$. 
\end{prop}
\begin{proof} Let us recall the way we obtain the colouring for $E(K_n)$ with $n=q^b$. 
Each edge $(i,j) \in E(K_n)$ with $i<j$ 
is assigned with the the colour $F \big( (\phi(i),\phi(j)) \big) $. 
Therefore, to show this colouring is bounded with high probability, 
it suffices to show that with high probability,
for a constant $C_0$, the following two statements holds.
 \begin{enumerate}[(i)]
 \item for any $Y_0\in \Bbb F_q^b$ and $Z_0\in \Bbb F_q^{2a}$, 
 the equation 
 \begin{equation}\label{equation_1}
 F(X,Y_0)= Z_0
 \end{equation}
 has at most $\frac{C_0}{2}$ solutions for $X\in \Bbb F_q^b$.
 \item for any $X_0\in \Bbb F_q^b$ and $Z_0\in \Bbb F_q^{2a}$, 
 the equation 
 \begin{equation}
 F(X_0,Y)= Z_0
 \end{equation}
 has at most $\frac{C_0}{2}$ solutions for $Y\in \Bbb F_q^b$.
 \end{enumerate}
 Clearly, we only need to check item (i) since (ii) can be dealt with very similarly. 
 We fix  $Y_0\in \Bbb F_q^b$ and $Z_0\in \Bbb F_q^{2a}$. Let $\mathcal W$ denote the set of $X\in \Bbb F_q^b$ satisfying 
  $F(X,Y_0)=Z_0$. We are interested in its size $|\mathcal W|$. 
  Note that $\mathcal W$ is a random variety over $\Bbb F_q$ contained in $\Bbb F_q^b$, that is, 
  $\mathcal W=\mathcal W(\Bbb F_q)$, and $\mathcal W$ has complexity at most $M$ which is the larger than 
    $a,b$, and $d$. We need to estimate the moments of $|\mathcal W|$.

For an integer $p\geq 0$,  for any $p$ distinct vertices $X_1,\cdots,X_p\in \Bbb F_q^{b}$, by Lemma~\ref{random_polynomial}, 
\begin{equation}
\text{\textbf{Prob}} \big[F(X_i,Y_0)=Z_0 \text{ for each } i=1,\cdots,p\big] =\frac{1}{q^{2a p}}.
\end{equation} 
  
For any positive integer $p$, 
the random variable $|\mathcal W|^p$ counts the number of $p$-tuples of solutions of (\ref{equation_1}). Note that these tuples might repeat themselves. Then we estimate its expected value summing over the possible number of actual number of solutions appeared in these tuples. Indeed, if the actual number of solutions appearing in this $p$-tuple is $j\in [p]$, there are $q^{bj}$ ways to choose this $i$ elements, and there are $j^p$ ways to make the choices of this $p$-tuple. So we have
\begin{equation}
\Bbb E(| \mathcal W|^p) \leq \sum_{j=1}^p q^{bj}\cdot j^p \frac{1}{q^{2aj}} \leq \sum_{j=1}^p  j^p =O_p(1),
\end{equation}
where we have used $2a\geq b$. Then by the Lang-Weil bound (Lemma~\ref{Langweil}), 
there exists a constant, say $\frac{C_0}{2}$, such that, 
we can take $p=2a+b+2$, and apply the Markov's inequality to make the following estimate.
\begin{align}
   & \ \ \text{\textbf{Prob}} \big [ |\mathcal W|\geq \frac{C_0}{2}\big ] = \text{\textbf{Prob}} \big [ |\mathcal W|\geq \frac{q}{2}\big ]\\
=  & \ \ \text{\textbf{Prob}} \big [ |\mathcal W|^p \geq (\frac{q}{2})^p\big ] \leq \frac{2^p \Bbb E(|\mathcal W|^p)}{q^p}\leq  \frac{1}{q^{2a+b+1}}.\nonumber
\end{align}
The choices of $Y_0,Z_0$ as in item (i) or the choices of $X_0,Z_0$ in item (ii) give us 
$2q^{2a+b}$ choices in total. 
By union bound, 
with probability $\frac{1}{q}$, for any choice of $(Y_0,Z_0)$ or $(X_0,Z_0)$ among those fixed, 
each of the equations in (i) or (ii) 
has at most $\frac{C_0}{2}$ solutions. It follows the colouring for the graph $G$ is $C_0/2$-bounded with probability at most $\frac 1q$. Then the colouring of $E(K_n)$ is $C_0$-bounded with high probability and the proof finishes. 
\end{proof}

\begin{proof}[Remains of the proof of Theorem~\ref{main}] For any $n$, one can choose some prime $q$ with $n^{1/b} \leq q\leq 2n^{1/b}$ by Bertrand's postulate. Then we regard $[n]$ as a subset of $\Bbb F_q^b$ and consider the random colouring with $q^{2a}$ colours as above. Note that 
$n^{\frac {2a}{b}} \leq q^{2a} \leq 2^{2a} n^{\frac{2a}{b}}$.
By Propositions~\ref{no_color_iso_pair} and \ref{bounded_whp}, since $2a\geq b$, when $n$ is sufficiently large so that $q$ is sufficiently large, 
with positive probability, the colouring admits no colour isomorphic pair of $T^{K_0}$'s and is $C_0$-bounded. 
Finally, we apply Lemma~\ref{bounded_lemma} to finish the proof. 
\end{proof}

\begin{proof}[Proof of Corollary~\ref{2a<b}]
With the assumption that $2a\leq b$, we claim that the balanced tree $T$ 
must contain some unrooted vertex which joins at least two roots. 
Suppose otherwise, then for any choice $v$ 
of the $a$ unrooted vertices of $T$, at most one of its neighbours can be a root. It follows 
the total number $r$ of the roots is at most $a$. In particular the total vertex number of $T$ is at most $2a$. 
On the other hand, since $2a \leq b$, we see that the total vertex number of the tree $T$ 
is at least $2a+1$, which is absurd. 
Therefore, $T$ contains a subgraph which is a copy of path of length $2$, namely $P_2$,
 whose two endpoints are both rooted. Note for this rooted $P_2$ which is clearly balanced, 
 we have exactly $\rho_{P_2}=2$. 
By Theorem~\ref{main}, for sufficiently large $t$, we have $f_2(n,P_2^{t})=O(n)$. Since $f_2$ is decreasing with respect to inclusion relation of graphs, 
we have $f_2(n,T^t)\leq f_2(n,P_2^t)=O(n)$. 
Finally, since any proper edge colouring of $K_n$ uses at least $n-1$ colours, we know $f_2(n,T^t)=\Omega(n)$
the conclusion follows. 
\end{proof}

\section{Lower bound for $f_2(n,K_{s,t}^{\lowercase{sub}})$}\label{lower_bounds_Kst'}
Viewing Theorem~\ref{main}, the proof of Theorem~\ref{f2_lowerbound} 
reduces to the following result.
\begin{theorem}\label{lower}
 For any $2\leq s\leq t$, $f_2(n,K_{s,t}^{\text{sub}})=\Omega(n^{1+\frac 1s})$.
\end{theorem}
The proof of this theorem is based on a general embedding scheme developed by Janzer~\cite{janzer2018improved}.
The main idea is as follows. Fix an arbitrary proper colouring of $E(K_n)$ 
with $k_0=o(n^{1+\frac 1s})$ colours and we will show there exists at least two colour isomorphic and vertex disjoint copies of $K_{s,t}^{\text{sub}}$ in $K_n$. 
Note that each colour class gives a monochromatic matching.
We choose only monochromatic $2$-matchings from them, to construct an auxiliary graph 
which already indicates pairs of colour isomorphic edges. 
Finally, we try to embed a copy of $K_{s,t}^{\text{sub}}$ into the auxiliary graph in a clean way, 
meaning that back to the original graph, there is a pair of colour isomorphic copies of vertex disjoint $K_{s,t}^{\text{sub}}$.

So we start by counting the total number of monochromatic $2$-matchings. 
Note each colour class $i$ 
with $C_i$ edges provides $C_i\choose 2$ such monochromatic $2$-matchings. 
Therefore, applying Jensen's inequality, 
the total number of two monochromatic $2$-matchings is at least 
\begin{equation}
\sum_{i=1}^{k_0} {C_i\choose 2} \geq k_0 { \frac {1}{k_0}\sum_{i}C_i \choose 2} \sim \frac{n^4}{k_0}= \omega(n^{3-\frac 1s}).
\end{equation}

We are going to define an auxiliary bipartite 
graph $\mathcal F$ as follows. 
Consider a random ordering of the vertex set of the edge coloured complete graph $K_n$, and simply label 
these vertices with $1,\cdots,n$ according this random order.
Then we partition the vertex set into four equal-sized parts as follows 
(up to throwing away a few vertices, we can 
assume $n$ is divisible by $4$ for simplicity).  
\begin{align}
  &  X_1=\{1,\cdots, \frac{n}{4}\},    \ \   X_2=\{\frac{n}{4}+1,\cdots, \frac{n}{2}\}, \\
  &  Y_1=\{\frac{n}{2}+1,\cdots, \frac{3n}{4}\},     Y_2=\{ \frac{3n}{4}+1,\cdots,n\}.
  \end{align} 
Define a bipartite graph $\mathcal F$ with the vertex bipartition 
$V(\mathcal F) = \big( X_1\times X_2 \big) \bigcup \big(Y_1\times Y_2\big)$. 
A pair $( (x_1,x_2), (y_1,y_2))$ belongs to edge set $E(\mathcal F)$
if and only if $(x_1,y_1)$ and $(x_2,y_2)$ form
 a monochromatic matching in the edge coloured graph $K_n$. 

Clearly for each labeled monochromatic $2$-matching, 
the probability that it is included in this collection is exactly $1/4^4$. 
It follows that the expected edge number of $\mathcal F$ satisfies that
\begin{equation}
\Bbb E(e (\mathcal F))=\omega(n^{3-\frac 1s})= \omega \big(v(\mathcal F)^{\frac 32-\frac {1}{2s}} \big).
\end{equation}
So we can find one specific ordering, with respect to which, the auxiliary graph $\mathcal F$ has 
at least $ \omega \big( v(\mathcal F)^{\frac 32-\frac {1}{2s}} \big)$
edges. 
We fix this ordering by simply writing the vertex set of $K_n$ as $[n]$.
Since it is known that $\text{ex}(n,K_{s,t}^{\text{sub}})= O(n^{\frac 32-\frac {1}{2s}})$ (see Lemma~\ref{Kst'turan}), 
it immediately follows that in the auxiliary graph, there exists a copy of $K_{s,t}^{\text{sub}}$. 
However, it does not automatically imply the existence of two colour isomorphic copies of $K_{s,t}^{\text{sub}}$ in the coloured graph $K_n$. 

\begin{definition}
We call a subgraph $H\subset \mathcal F$ \textit{clean}, if its vertex set $V(H)$ corresponds to exactly 
$2v(H)$ vertices in $K_n$. In other words, the underlying vertices corresponding to all the elements in $V(H)$ 
are pair-wise disjoint. 
\end{definition}

With this definition, in order to find a pair of colour isomorphic and vertex disjoint copies of $K_{s,t}^{\text{sub}}$ in $K_n$,
it suffices to find a clean copy of $K_{s,t}^{\text{sub}}$ in the auxiliary graph $\mathcal F$.
We have the following crucial observation.
\begin{lemma}[see Lemma 2.2 of \cite{xu2020color}]\label{not_many_intersection}
Suppose $e\in V(\mathcal F)$ and $f,f'$ are two neighbours of $e$ in $\mathcal F$. 
Then the three elements $e,f,f'$ of $V(\mathcal F)$
do not share any common underlying vertices in $K_n$.
\end{lemma} 
\begin{proof}
We can assume $(e,f), (e,f')\in E(\mathcal F)$ and the other case is similar. 
Write $e=(x_1,x_2)$, $f=(y_1,y_2)$ and $f'=(y_1',y_2')$. The only possibility of some common underlying vertex
is between $f$ and $f'$. Suppose it happens, then either $y_1=y_1'$ or $y_2=y_2'$. 
Suppose the first case happens and the second case can be dealt with similarly. It follows 
that the two edges $(x_2,y_2)$ and $(x_2,y_2')$ was assigned with the same colour with the colour of 
the edge $(x_1,y_1)=(x_1,y_1')$, contradicting the colouring being proper. 
\end{proof}

It is standard to extract a regular bipartite graph subgraph $\mathcal {F}'$ of $\mathcal F$.  
More precisely, apply Lemma~\ref{almost_regular}. 
There is a subgraph $\mathcal F'$ of $\mathcal F$
which is bipartite on the vertex bipartition $\mathcal A \bigcup \mathcal B$, 
with $|\mathcal A|=m$ and $|\mathcal B|=\Theta(m)$, where $m$ tends to infinity when $v(\mathcal F)$ tends to infinity. 
Moreover, the minimal degree satisfies
\begin{equation}\label{degree_condition}
\delta= \delta(\mathcal F')=\omega(m^{\frac{1}{2}-\frac{1}{2s}}).
\end{equation} 
Every vertex of $\mathcal F'$ has degree at most $\Delta \delta$ for a constant $\Delta$. 

Now we form a coloured graph on the vertex set $\mathcal A$.
\begin{definition}
A coloured graph $\mathcal F'_C(\mathcal A)$ is obtained as follows. 
For any two elements $E_1, E_2\in\mathcal A$, we say they form a 
\textit{red edge}
if their common neighbourhood in $\mathcal B$ has size at least $2st$.
We say they form a \textit{blue edge} if their common neighbourhood in $\mathcal B$ has size at least $1$ and at most $2st-1$. 
The rest of the pairs are left \textit{uncoloured}.
\end{definition}

\begin{lemma}\label{red_means_clean} 
Suppose there is a red copy of $K_{s,t}$ in the coloured graph $\mathcal F'_C(\mathcal A)$. Assume as 
an independent subset of the graph $\mathcal F'$, it is clean. 
Then based on it, one can embed a clean copy of $K_{s,t}^{\text{sub}}$ in $\mathcal F'$. \end{lemma}

\begin{proof} 
The proof is a greedy embedding process of a clean copy of $K_{s,t}^{\text{sub}}$ in $\mathcal F'$. 
For each red pair in the red and clean copy of $K_{s,t}$, 
we try to choose an element from their common neighbourhood  in $\mathcal B$, and we need to ensure the resulting copy of $K_{s,t}^{\text{sub}}$ is clean. 
Let us list the red pairs as $P_1,\cdots,P_{st}$. 

For $P_1$, we choose arbitrarily an element $B_1$ which is in the common neighbourhood of the pair $P_1$ of $\mathcal A$ elements. 
Inductively, suppose we have already chosen $B_1,\cdots, B_k \in \mathcal B$, $k<st$, whose underlying vertices are pair-wise distinct, and each $B_i$ is the common neighbourhood of the pair $P_i$ of $\mathcal A$ elements. 
Then we consider the common neighbourhood of the pair $P_{k+1}$, which has size 
at least $2st > 2k$.  
Due to Lemma~\ref{not_many_intersection}, there are at most $2k$ elements which contain at least  one of the underlying vertices of $B_1,\cdots, B_k$. 
 Excluding these choices, we still have choices available. 
Therefore, we can choose one such element as $B_{k+1}$ and continue the induction step. The induction process terminates
at $k=st$ and we have found a clean copy of $K_{s,t}^{\text{sub}}$.
\end{proof}

For the following several lemmas, we suppose there is no clean copy of $K_{s,t}^{\text{sub}}$ found in $\mathcal F'$. 
\begin{lemma}\label{Janzer_blue} 
For any element $B\in \mathcal B$, 
among its neighbourhood $N_{\mathcal F'}(B)$ one can find at least 
$\Omega(\delta^2)$ many blue edges of $\mathcal F'_C(\mathcal A)$.
Consequently, for every $A\in \mathcal A$, 
its blue neighbourhood in $\mathcal F'_C(\mathcal A)$ has size $|N_b(A)|= \Theta(\delta^2)$.
\end{lemma}
\begin{proof}
For any vertex $B \in \mathcal B$, its neighbourhood is denoted by $N_{\mathcal F'}(B)$, which induces a subgraph in $\mathcal F'_C(\mathcal A)$. This coloured subgraph has edge number $\Omega(\delta^2)$ and each edge is either red or blue. 
There can not exist a red copy of $K_{s,t}$, because otherwise one finds a clean copy of $K_{s,t}^{\text{sub}}$
by Lemma~\ref{red_means_clean}.
Therefore, Tur\'{a}n's theorem ensures there are at least $\Omega(\delta^2)$ blue edges found within this subgraph. 
Now for any $A\in \mathcal A$, its neighbourhood in $\mathcal F'$ has size at least $\delta$. 
Note also the vertices of each blue edge in $\mathcal F'_C(\mathcal A)$ have at most $st-1$ common neighbours in $\mathcal B$.
By double counting, we see these elements in $\mathcal B$ 
provide at least $\Theta(\delta^2)$ blue edges. 
On the other hand, the number of 
blue neighbours of $A\in \mathcal A$ is at most $O(\delta^2)$ 
since the graph $\mathcal F'$ is almost regular. Thus, it has order $\Theta(\delta^2)$.
\end{proof}

\begin{definition}
For any $X_1,\cdots,X_s\in \mathcal A$, define  $\mathcal P(X_1,\cdots,X_s)$ as the set of $Y\in \mathcal A$ so that 
there exists $B_i\in \mathcal N_{\mathcal F'}(X_i,Y)$ for each $i=1,\cdots,s$, 
and they together form a clean copy of $K_{s,1}^{\text{sub}}$.
\end{definition}

\begin{lemma} Summing over all $(X_1,\cdots,X_s) \in \mathcal A^s$ 
which form a clean independent set in $\mathcal F'$, we have 
\begin{equation} \sum |\mathcal P(X_1,\cdots,X_s)|=\Omega(m\delta^{2s}).
\end{equation}
In words, the total number of tuples 
$(X_1,\cdots,X_s, Y)$ for which one can choose one common neighbour for each 
$(X_i,Y)$ for all $i \in \{1,\cdots, s\}$ to form a clean copy of $K_{s,1}^{\text{sub}}$ is at least $\Omega(m\delta^{2s})$.
\end{lemma}
\begin{proof} By previous lemma, for each $Y\in \mathcal A$, its blue neighbourhood $N_b(Y)$ 
in $\mathcal F'_C(\mathcal A)$
has size $\Theta(\delta^2)$. Choose $X_1 \in N_b(Y)$ arbitrarily in $\Theta(\delta^2)$ ways, with some $B_1$ which is 
a common neighbour of $Y$ and $X_1$ in $\mathcal F'_C(\mathcal A)$. 
For any $B\in N_{\mathcal F'}(Y) \backslash \{B_1\}$, in its neighbourhood there are at most $2$ elements whose underlying vertices intersect the two $K_n$ vertices underlying the element $X_1$. 
So the possible choices for $X_2$ are at least 
$\Theta(\delta^2)-\Theta(\delta)=\Theta(\delta^2)$.
Inductively, suppose $k<s$, 
we have chosen $X_1,\cdots,X_k$ and $B_i\in N_{\mathcal F'}(X_i) \bigcap N_{\mathcal F'}(Y)$ for $i=1,\cdots,k$, 
so that they together form a clean copy of $K_{k,1}^{\text{sub}}$. Then for any 
$B\in N_{\mathcal F'}(Y) \backslash \{B_1,\cdots, B_k\}$, at most $2k$ of its neighbours share 
one underlying vertex with one of the elements $X_i$, for $i=1,\cdots,k$. 
Excluding these, we have $\Theta(\delta^2)$ choices for $X_{k+1}$ together 
with some $B_{k+1} \in N_{\mathcal F'}(Y)\bigcup N_{\mathcal F'}(X_{k+1})$. 
We finish after $s$ steps.
In total we have $\Theta(\delta^{2s})$ choices for the tuples $(X_1,\cdots,X_s)$, 
each of which gives an ordered 
clean copy of $K_{s,1}^{\text{sub}}$.
Therefore, we see for each $Y \in \mathcal A$ fixed, 
there are at least $\Theta(\delta^{2s})$ many $(s+1)$-tuples, namely $(X_1,\cdots,X_s,Y)$, which can induce 
clean copies of $K_{s,1}^{\text{sub}}$'s in $\mathcal F$. 
Summing over $Y$ the conclusion follows. 
\end{proof}

\begin{lemma}\label{BX_k_manychoices}
For any $s$-tuple $X_1,\cdots,X_s\in \mathcal A$ with $|\mathcal P(X_1,\cdots,X_s)| =\omega(1)$, 
there exists some $B\in \mathcal B$ and some $X_k$ in this list,
 a subset $\mathcal Y\subset \mathcal P(X_1,\cdots,X_s)$ of size 
$|\mathcal Y|=\Omega( |\mathcal P(X_1,\cdots,X_s)|)$, so that 
$\mathcal Y\bigcup \{ X_k \} \subset N_{\mathcal F'}(B)$.
\end{lemma}
\begin{proof} For $\mathcal P= \mathcal P(X_1,\cdots,X_s)$, choose a maximal set $\{Y_1,\cdots,Y_\ell\}\subset \mathcal P$
so that one can choose $B_{ij}$ for each $1\leq i\leq s, 1\leq j\leq \ell$, so that with $X_1,\cdots,X_s$ together, they form a
clean copy of $K_{s,\ell}^{\text{sub}}$. By assumption it follows $\ell<t$. 
Now for any $Y\mathcal \in \mathcal P \backslash \{Y_1,\cdots,Y_\ell\}$, 
there exists some $X_k$ and some $B_{ij}$ 
from the above list, so that 
$N_{\mathcal F'}(Y) \bigcap N_{\mathcal F'}(X_k)$ and $\{B_{ij}\}$ share one or two underlying vertices. 
 Since there are only $s^2\ell$ choices of the pair $(X_k,B_{ij})$, 
it follows some specific choice of $(X_k,B_{ij})$ 
has been chosen for $\Omega(|\mathcal P|)$ times. Since among $N_{\mathcal F'}(X_k)$ there are at most two 
elements which shares one underlying vertex with $B_{ij}$. So we can choose $B$ which share one underlying vertex with 
$B_{ij}$, and a set $\mathcal Y\subset \mathcal P$ with 
$|\mathcal Y|=\Omega(|\mathcal P|)$, so that $\mathcal Y\bigcup \{X_k\} \subset N_{\mathcal F'}(B)$.
\end{proof}

\begin{lemma}\label{non_red_Y} From the set $\mathcal Y$ obtained from the previous lemma, 
the number of $s$-tuples $(Y_1,\cdots,Y_s)\in \mathcal Y^s$ which are pairwise non-red in 
$\mathcal F'_C(\mathcal A)$
and form a clean independent set in $\mathcal A$ is at least $\Omega(|\mathcal P(X_1,\cdots,X_s)|^s)$.
\end{lemma}
\begin{proof} Note the set $\mathcal Y$ shares the same common neighbour $B\in \mathcal B$, so any two elements in 
$\mathcal Y$ do not share any same underlying vertex in $K_n$. So we only need to worry about the non-red 
property. This is from a standard double counting. Recall that 
the coloured graph $\mathcal F'_C(\mathcal A)$ contains no red copy of complete graph $K_{s+t}$. 
Thus, consider the Ramsey number $R(s+t,s)$ which is the smallest integer $n_0$ such that any coloured $K_{n_0}$ either contains an $(s+t)$ red clique or an $s$ non-red clique. It follows that every subset of size $R(s+t,s)$ must contain a copy of non-red $s$-tuple. Now each non-red $s$-tuple is counted for $|\mathcal Y| - s\choose R(s+t,s)-s$ times. So we can find at least 
$ {|\mathcal Y| \choose R(s+t,s)} / {|\mathcal Y| - s\choose R(s+t,s)-s}$ many such $s$-tuples, which is 
$\Omega(|\mathcal Y|^s)=\Omega(|\mathcal P(X_1,\cdots,X_s)|^s)$, by Lemma~\ref{BX_k_manychoices}.
\end{proof}

Then, we can summarise what can be obtained assuming that $\mathcal F'$ contains no 
clean copy of $K_{s,t}^{\text{sub}}$.
\begin{prop}\label{big_light_neighbour}
 Suppose $\mathcal F'$ contains no clean copy of 
$K_{s,t}^{\text{sub}}$. Then 
there exists distinct elements 
$E_1,\cdots, E_s\in \mathcal A$ with the following properties. 
\begin{enumerate}[(a)]
\item $\{E_1,\cdots,E_s\}$ is a clean independent set in $\mathcal F'$.
\item $E_1,\cdots,E_s$ are pairwise non-red in $\mathcal F'_C(\mathcal A)$.
\item the common blue neighbourhood of this $s$-tuple in $\mathcal F'_C(\mathcal A)$, 
denoted by $N_b(E_1,\cdots,E_s)$, has size at least $\omega(\delta)$.
\end{enumerate}
\end{prop}

\begin{proof}
By previous lemma, with the assumption that $\mathcal F'$ contains no clean copy of 
$K_{s,t}^{\text{sub}}$, we are going to count the total number of \textit{good} $2s$-tuples, i.e., the size of 
the collection $\mathcal S$ of $(X_1,\cdots,X_s, Y_1,\cdots,Y_s)$'s with the following properties. 
\begin{enumerate}
\item $\{Y_1,\cdots, Y_s\}$ forms a clean independent set in $\mathcal F'$ and they are pair-wise non-red in 
$\mathcal F'_C(\mathcal A)$.
\item $Y_i\in N_b(X_j)$ for any $i, j =1,\cdots, s$.
\item there exists $X_k$ in this list, and some $B\in \mathcal B$, so that $\{Y_1,\cdots,Y_s,X_k\} \subset N_{\mathcal F'}(B)$.
\end{enumerate}
Applying Lemma~\ref{non_red_Y}, using Jensen's inequality, and noticing the degree condition (\ref{degree_condition}), 
the total number of such $2s$-tuples is at least
\begin{equation}
              \sum  \Omega(|\mathcal P(X_1,\cdots,X_s)|^s) \geq 
\Omega \big( m^s (\frac{m\delta^{2s}}{m^s})^s\big)=\omega(m\delta^{2s}),
\end{equation} 
where the sum is taken over all $s$-tuples $(X_1,\cdots,X_s)$. 
On the other hand, we can choose $B$ in $m$ ways. Then based on $B$, 
we can choose $k$ from $\{1,\cdots, s\}$ and then choose $Y_1,\cdots,Y_s,X_k \in N_{\mathcal F'}(B)$ 
in $O(\delta^{s+1})$ ways. Thus, there are still $\omega(\delta^{s-1})$ ways to extend this to a good $2s$-tuple. 
Note the number of coordinates to decide is $s-1$. 
Therefore, for a certain fixed choice of $X_k$ and $(Y_1,\cdots,Y_s)=(E_1,\cdots,E_s)$, 
there are at least $\omega(\delta)$ distinct elements in $\mathcal A$,
each of which is a common blue neighbour of $E_1,\cdots, E_s$, which is the porperty (c). Note by choice, $(E_1,\cdots,E_s)$ satisfies properties (a) and (b). The proposition follows. 
\end{proof}

\begin{proof}[End of proof of Theorem~\ref{lower}] 
Suppose for contradiction that the coloured graph $K_n$ contains no two colour isomorphic and vertex disjoint
 copies of $K_{s,t}^{\text{sub}}$. 
Then we obtain the bipartite graph $\mathcal F'$ 
on the vertex bipartition $\mathcal A \bigcup \mathcal B$, as explained earlier in this section. 
Then it follows $\mathcal F'$ does not contain a clean copy of $K_{s,t}^{\text{sub}}$.
Apply Proposition~\ref{big_light_neighbour}, and obtain elements $E_1,\cdots, E_s\in \mathcal A$, with the three properties stated there. In particular, the non-red property (b) implies that for any the pair $(E_i,E_j)$, 
their common neighbourhood in $\mathcal B$ has size 
at most $2st$. 

Choose any element $F_1\in N_b(E_1,\cdots,E_s)$.
For any blue edge $(E_i,F_1)$, choose one element $B_{i,1}$ from their common neighbourhood in $\mathcal B$. 
Since $\{B_{i,1}\}_{i=1}^s$ are all neighbours of $F_1$ in $\mathcal F'$, 
they must have pair-wise disjoint underlying vertices in $K_n$. Also, the underlying pair of vertices of $F_1$ does not intersect any of the underlying vertices for any $E_i$ above, since $E_i$ and $F_1$ have a common neighbour $B_{i,1}$. Thus, we have found 
a clean copy of $K_{s,1}^{\text{sub}}$ so far. 

Inductively, suppose we have found  $F_1,\cdots,F_k \in \mathcal A$ and $B_{i,j} \in N_{\mathcal F'}(E_i) \bigcap N_{\mathcal F'}(F_j)$, with 
$1\leq i\leq s, 1\leq j \leq k$, 
so that they together with $E_1,\cdots,E_s$ form a clean copy of $K_{s,k}^{\text{sub}}$, where each $B_{i,j}$ is a common neighbour of $E_i$ and $F_j$ in $\mathcal F'$. 
To finish the induction step, we look for $F_{k+1}$ from the $N_b(E_1,\cdots,E_s)$ and then choose $B_{i,k+1}$ 
from the common neighbourhood of $E_i$ and $F_{k+1}$, for each $i=1,\cdots,s$, 
which helps to create a clean copy of $K_{s,k+1}^{\text{sub}}$.  

Now the union of all the already embedded elements in $\mathcal A$ has size $sk$. 
We need to consider two kinds of pairs. Firstly, any pair $(E_i,E_j)$ is non-red in $\mathcal F'_C(\mathcal A)$ by 
Proposition~\ref{big_light_neighbour}. Secondly, any pair $(E_i,F_j)$ is a blue edge in $\mathcal F'_C(\mathcal A)$. 
So both pairs have their common neighbourhood of size at most $2st$. So their union is of size at most 
$2st \times ( {s \choose 2} + sk)$ which is a constant. It follows that, 
the union of all the neighbourhoods in $\mathcal A$ of all these elements has size at most $O(\delta)$.
Excluding these elements from $N_b(E_1,\cdots,E_s)$, then for any choice of $F_{k+1}$, we can choose arbitrarily an element 
$B_{i,k+1}\in N_{\mathcal F'}(E_i,F_{k+1})$. They are different with any previous fixed $B_{i,j}$ because they do not belong to the common neighbourhood of any pair $(E_i,F_j)$. They are different from each other because they do not belong to the common neighbourhood of any pair $(E_i,F_j)$. 

However, there is one more obstruction for a clean copy of $K_{s,k+1}^{\text{sub}}$. 
This obstruction is the situation that for $i'\neq i$ and for some $j\leq k$, 
the elements $B_{i',k+1}$ and $B_{i,j}$ share some common underlying vertex in $K_n$.
Note the common neighbourhood of $E_i$ and $F_j$ has size at most $2st$, which involves at most $4st$ underlying vertices. It follows that for the neighbourhood of the element $E_{i'}$,
there are at most $4st$ elements whose underlying vertices intersect these elements. 
Counting over all the possible choices of $E_i,E_{i'},F_j$, we again obtain that there are at most $O(\delta)$ 
vertices in $\mathcal A$ which might cause such a situation. 
Therefore, noticing (c) of Proposition~\ref{big_light_neighbour}, 
we can exclude these elements from $N_b(E_1,\cdots,E_s)$ and choose any element $F_{k+1}$ from the leftover vertices. 
Then the induction step was done and we obtain a clean copy of $K_{s,k+1}^{\text{sub}}$.

Finally, the induction stops when we reach $k+1=t$ and we have embedded a clean copy of 
$K_{s,t}^{\text{sub}}$ in $\mathcal F'$. It means that we have found a pair of colour isomorphic and vertex disjoint copies of $K_{s,t}^{\text{sub}}$ in $K_n$, which is a contradiction. The proof is complete. 
\end{proof}

\section{Some Further Discussions}\label{final}
The function $f_2(n, H)$ is related with the classical extremal number $\text{ex}(n,H)$ of a bipartite graph. 
We first note that, for the $1$-subdivision of complete graph $K_t^{\text{sub}}$, in~\cite{janzer2018improved}, Janzer proved that 
$\text{ex}(n, K_t^{\text{sub}})=O(n^{\frac32-\frac{1}{4t-6}})$.
More recently in~\cite{xu2020color}, Ge and Xu showed the lower bound $f_2(n,K_t^{\text{sub}})=\Omega(n^{1+\frac{1}{2t-3}})$. Comparing the above two results, it is plausible the following hold. \\

\noindent{\textbf{Problem A}}.
For a bipartite graph $H$ and for $\varepsilon \in [0,\frac 12]$, 
is it true that for any $C_1>0$ there exists $C_2>0$, so that, 
\begin{enumerate}[(a)]
\item $\text{ex}(n, H) \leq C_2 n^{\frac 32-\varepsilon}$ implies $f_2(n,H) \geq C_1 n^{1+2\varepsilon}$?
\item Or equivalently, 
$f_2(n,H) \leq C_1 n^{1+2\varepsilon}$  implies $\text{ex}(n,H)\geq C_2 n^{\frac 32-\varepsilon}$?
 \end{enumerate}\
 \\
Formulation (a) is trivially satisfied by the four cycle $C_4$ by taking $\varepsilon=0$ and 
some appropriate choices of $C_1$ and $C_2$.
More generally, (a) is also confirmed for all even cycles (see Theorem 1.11 of~\cite{janzer2020rainbow}).
Let us look at the family of large rooted powers of balanced trees. 
Problem A also helps to connect Theorem~\ref{main} with the recent breakthrough result towards the famous rational exponent conjecture in~\cite{bukh2018rational}.
More precisely, recall the following result.
\begin{theorem}[Lemma 1.2 of~\cite{bukh2018rational}]\label{rational}
Let $T$ be a balanced rooted tree with $\rho_T=\frac ba$, 
and let 
$T^{K_0}$ be its $K_0$-th rooted power. 
Then for sufficiently large $K_0$,
\begin{equation}
\text{ex}(n,T^{K_0})=\Omega(n^{2-\frac{1}{\rho_T}}).
\end{equation}
\end{theorem}
Therefore, for large rooted powers of a balanced tree $T$
with $\rho_T=\frac{b}{a}$, where $2a\geq b$ and $K_0$ is sufficiently large, we have 
$\text{ex}(n,T^{K_0})= \Omega(n^{2-\frac{1}{\rho_T}})=\Omega(n^{2-\frac ab})= \Omega(n^{\frac 32-\frac{2a-b}{2b}})$.
Then we note also that $n^{\frac{2a}{b}}=n^{1+\frac{2a-b}{b}}$.
Thus, in order to answer Problem A, formulation (a) in the affirmative for all large rooted powers of balanced trees, 
it is equivalent with finding matching lower bounds of Theorem~\ref{main} for these graphs. 
Note this matching lower bound was obtained in the special case of the theta graphs $\theta_{3,\ell}$ (i.e., the graph consisting $\ell$ internally vertex distinct paths of length $3$ connecting two fixed vertices, See Section 5 of~\cite{conlon2021repeated}), and we have confirmed this for $1$-subdivisions of complete bipartite graphs 
in Theorem~\ref{f2_lowerbound}.

Another comment is as follows. 
If Problem A can be solved for all large powers of rooted balanced trees, 
then combined with a family version of Theorem~\ref{main}, 
one can also deduce a family version of the rational exponent conjecture for the function $f_2(n,H)$
 (for the graph version formulation, see Conjecture 3.1 of~\cite{xu2020color}).

Now we look at Problem A, formulation (b).
Recall (\ref{2v-2_thm}) shows the relation between edge and vertex numbers has implications for upper bounding the $f_2$ function.
If Problem A, formulation (b) is established, one would also see that, the edge/vertex relation of a graph would give 
some lower bounds of its extremal numbers. 
Therefore, we take one step further, posing the following conjecture. \\

\noindent{ \textbf{Conjecture B.}} If $H$ is a bipartite graph containing a cycle, then 
\begin{equation}
f_2(n,H) = \left\{
        \begin{array}{ll}
            O(n), &  \text{ if } e(H)\geq 2v(H)-4,\\
            O\big(n^{\frac{2v(H)-4}{e(H)}}\big), & \text{ if } e(H)\leq 2v(H)-4.
        \end{array}
    \right.
\end{equation}
The open problem that whether or not $f_2(n,C_4)=O(n)$ would obtain an affirmative answer following Conjecture B.
 Moreover, if one assumes that Problem A has an affirmative answer, 
 Conjecture B would imply the famous conjecture that for all even cycle $C_{2\ell}$, 
$\text{ex}(n,C_{2\ell})=\Theta (n^{1+\frac{1}{\ell}})$. 
 
\section*{ACKNOWLEDGMENTS}
The authors thank Zixiang Xu for introducing the colour isomorphic problem to them. 
X-C. Liu is supported by Fapesp P\'os-Doutorado grant (Grant Number  2018/03762-2). 
\bibliographystyle{plain}
\addcontentsline{toc}{chapter}{Bibliography}
\bibliography{coloriso}
\end{document}